\documentclass[11pt]{article}
\usepackage[noadjust]{cite}
\usepackage[title]{appendix}
\usepackage{etex}
\usepackage{xcolor}
\usepackage{amsthm}
\usepackage{amsfonts}
\usepackage{amssymb}
\usepackage{amsgen}
\usepackage{amsmath}
\usepackage{amsopn}
\usepackage{verbatim}
\usepackage{xypic}
\usepackage{pgf}
\usepackage{xspace}
\usepackage{multicol}
\usepackage{makeidx}
\usepackage{eepic}
\usepackage{upref}
\usepackage{pgf}
\usepackage{tikz}
\usepackage[normalem]{ulem}
\usepackage{shuffle,yfonts}
\DeclareFontFamily{U}{shuffle}{}
\DeclareFontShape{U}{shuffle}{m}{n}{ <-8>shuffle7 <8->shuffle10}{}

\newcommand{\nc}{\newcommand}

\nc\ta{{\texttt{x}_0}}
\nc\tb{{\texttt{x}_1}}
\nc\tc{{\texttt{x}_{-1}}}
\nc{\tz}{\tilde\zeta}
\nc{\AMZV}{\mathsf {AMZV}}
\nc{\ud}{\mathrm{d}}
\nc{\ES}{\mathsf {ES}}
\nc{\MZV}{\mathsf {MZV}}
\nc{\MtV}{\mathsf {MtV}}
\nc{\MTV}{\mathsf {MTV}}
\nc{\MSV}{\mathsf {MSV}}
\nc{\MMV}{\mathsf {MMV}}
\nc{\MMVo}{\mathsf {MMVo}}
\nc{\MMVe}{\mathsf {MMVe}}
\nc{\AMMV}{\mathsf {AMMV}}
\nc{\AMTV}{\mathsf {AMTV}}
\nc{\AMtV}{\mathsf {AMtV}}
\nc{\AMSV}{\mathsf {AMSV}}
\nc{\CMZV}{\mathsf {CMZV}}
\nc{\sha}{\shuffle}
\nc{\cst}{\rotatebox[origin=c]{180}{$\sha$}}
\nc{\cstt}{\rotatebox[origin=c]{180}{$\scriptstyle \sha$}}
\nc{\de}{\delta}
\nc{\DD}{{\mathbb D}}
\nc{\anbb}[1]{\left\langle#1\right\rangle}
\nc{\bibb}[1]{\left\{#1\right\}}
\nc{\mibb}[1]{\left[#1\right]}
\nc{\smbb}[1]{\left(#1\right)}
\nc{\doubb}[1]{\llbracket#1\rrbracket}
\nc{\dm}[1]{\left|#1\right|}

\nc{\Gbinom}[2]{\genfrac{(}{)}{0mm}{0}{#1}{#2}}
\nc{\gbinom}[2]{\genfrac{(}{)}{0mm}{1}{#1}{#2}}
\nc{\Rbinom}[2]{\genfrac{\langle}{\rangle}{0mm}{0}{#1}{#2}}
\nc{\rbinom}[2]{\genfrac{\langle}{\rangle}{0mm}{1}{#1}{#2}}
\nc{\Qbinom}[2]{\genfrac{[}{]}{0mm}{0}{#1}{#2}_q}
\nc{\qbinom}[2]{\genfrac{[}{]}{0mm}{1}{#1}{#2}_q}

\nc{\binq}[2]{\genfrac{[}{]}{0mm}{0}{#1}{#2}}
\nc{\tbnq}[2]{\genfrac{[}{]}{0mm}{1}{#1}{#2}}
\nc{\cinq}[2]{\genfrac{\{}{\}}{0mm}{0}{#1}{#2}}
\nc{\tcnq}[2]{\genfrac{\{}{\}}{0mm}{1}{#1}{#2}}

\nc{\mfrac}[2]{\genfrac{}{}{0pt}{}{#1}{#2}}
\nc{\tf}{\tfrac}
\nc{\db}{{\mathbb D}}
\nc{\pari}{{\rm par}}
\nc{\dk}{{\mathbb K}}
\nc{\ola}{\overleftarrow}
\nc{\ora}{\overrightarrow}
\nc{\lra}{\longrightarrow}
\nc{\Lra}{\Longrightarrow}
\nc\Res{{\rm Res}}
\nc\setX{{\mathsf{X}}}
\nc\fA{{\mathfrak{A}}}
\nc\evaM{{\texttt{M}}}
\nc\evaML{{\text{\em{\texttt{M}}}}}
\nc\z{{\texttt{z}}}
\nc\emz{\emph{\texttt{z}}}
\nc\tx{{\texttt{x}}}
\nc\txp{{\tx_1}} % textstyle x positive 1
\nc\txn{{\tx_{-1}}} % textstyle x negative 1
\nc\neo{{1}}
\nc{\yi}{{1}}
\nc\one{{-1}}
\nc\gD{{\Delta}}
\nc\eps{{\varepsilon}}
\nc{\bfMB}{{\bf MB}}
\nc{\bftB}{{\bf tB}}
\nc{\bfTB}{{\bf TB}}
\nc{\bfSB}{{\bf SB}}
\nc{\bfB}{{\bf B}}
\nc{\bfp}{{\bf p}}
\nc{\bfq}{{\bf q}}
\nc{\bfr}{{\bf r}}
\nc{\bfu}{{\bf u}}
\nc{\bfv}{{\bf v}}
\nc{\bfa}{{\bf a}}
\nc{\bfw}{{\bf w}}
\nc{\bfy}{{\bf y}}
\nc{\T}{\ddot{t}}
\nc{\bfe}{{\boldsymbol{\sl{e}}}}
\nc{\bfi}{{\boldsymbol{\sl{i}}}}
\nc{\bfj}{{\boldsymbol{\sl{j}}}}
\nc{\bfk}{{\boldsymbol{\sl{k}}}}
\nc{\bfl}{{\boldsymbol{\sl{l}}}}
\nc{\bfm}{{\boldsymbol{\sl{m}}}}
\nc{\bfn}{{\boldsymbol{\sl{n}}}}
\nc{\bfs}{{\boldsymbol{\sl{s}}}}
\nc{\bft}{{\boldsymbol{\sl{t}}}}
\nc{\bfx}{{\boldsymbol{\sl{x}}}}
\nc{\bfz}{{\boldsymbol{\sl{z}}}}
\nc\bfgs{{\boldsymbol \gs}}
\nc\bfgl{{\boldsymbol \lambda}}
\nc\bfsi{{\boldsymbol \gs}}
\nc\bfet{{\boldsymbol \eta}}
\nc\bfeta{{\boldsymbol \eta}}
\nc\bfeps{{\boldsymbol \eps}}
\nc\mmu{{\boldsymbol \mu}}
\nc\bfone{{\bf 1}}
\nc{\myone}{{1}}

 \nc{\calA}{{\mathcal A}}
 \nc{\calB}{{\mathcal B}}
 \nc{\calC}{{\mathcal C}}
 \nc{\calD}{{\mathcal D}}
 \nc{\calE}{{\mathcal E}}
 \nc{\calF}{{\mathcal F}}
 \nc{\calG}{{\mathcal G}}
 \nc{\calH}{{\mathcal H}}
 \nc{\calI}{{\mathcal I}}
 \nc{\calJ}{{\mathcal J}}
 \nc{\calK}{{\mathcal K}}
 \nc{\calL}{{\mathcal L}}
 \nc{\calM}{{\mathcal M}}
 \nc{\calN}{{\mathcal N}}
 \nc{\calO}{{\mathcal O}}
 \nc{\calP}{{\mathcal P}}
 \nc{\calQ}{{\mathcal Q}}
 \nc{\calR}{{\mathcal R}}
 \nc{\calS}{{\mathcal S}}
 \nc{\calT}{{\mathcal T}}
 \nc{\calU}{{\mathcal U}}
 \nc{\calV}{{\mathcal V}}
 \nc{\calW}{{\mathcal W}}
 \nc{\calX}{{\mathcal X}}
 \nc{\calY}{{\mathcal Y}}
 \nc{\calZ}{{\mathcal Z}}
  \nc{\cala}{{\mathcal a}}
 \nc{\calb}{{\mathcal b}}
 \nc{\calc}{{\mathcal c}}
 \nc{\cald}{{\mathcal d}}
 \nc{\cale}{{\mathcal e}}
 \nc{\calf}{{\mathcal f}}
 \nc{\calg}{{\mathcal g}}
 \nc{\calh}{{\mathcal h}}
 \nc{\cali}{{\mathcal i}}
 \nc{\calj}{{\mathcal j}}
 \nc{\calk}{{\mathcal k}}
 \nc{\call}{{\mathcal l}}
 \nc{\calm}{{\mathcal m}}
 \nc{\caln}{{\mathcal n}}
 \nc{\calo}{{\mathcal o}}
 \nc{\calp}{{\mathsf p}}
 \nc{\calq}{{\mathcal q}}
 \nc{\calr}{{\mathcal r}}
 \nc{\cals}{{\mathcal s}}
 \nc{\calt}{{\mathcal t}}
 \nc{\calu}{{\mathcal u}}
 \nc{\calv}{{\mathcal v}}
 \nc{\calw}{{\mathcal w}}
 \nc{\calx}{{\mathcal x}}
 \nc{\caly}{{\mathcal y}}
 \nc{\calz}{{\mathcal z}}
 \nc{\ot}{{\otimes}}
\usetikzlibrary{arrows,shapes,chains}

\catcode`!=11
\let\!int\int \def\int{\displaystyle\!int}
\let\!lim\lim \def\lim{\displaystyle\!lim}
\let\!sum\sum \def\sum{\displaystyle\!sum}
\let\!sup\sup \def\sup{\displaystyle\!sup}
\let\!inf\inf \def\inf{\displaystyle\!inf}
\let\!cap\cap \def\cap{\displaystyle\!cap}
\let\!max\max \def\max{\displaystyle\!max}
\let\!min\min \def\min{\displaystyle\!min}
\let\!frac\frac \def\frac{\displaystyle\!frac}
\catcode`!=12

\nc{\gam}{{\gamma}}
\nc{\GG}{{\mathbb G}}
\nc{\PP}{{\mathbb P}}
\nc{\gG}{{\Gamma}}
\nc{\om}{{\omega}}
\nc{\vep}{{\varepsilon}}
\nc{\ga}{{\alpha}}
\nc{\gl}{{\lambda}}
\nc{\gb}{{\beta}}
\nc{\gd}{{\delta}}
\nc{\gf}{{\varphi}}
\nc{\gs}{{\sigma}}
\nc{\gk}{{\kappa}}
\nc{\gS}{\Sigma}
\let\oldsection\section
\renewcommand\section{\setcounter{equation}{0}\oldsection}

\allowdisplaybreaks

\DeclareMathOperator{\Li}{Li}

\DeclareMathOperator{\dch}{dch}
\DeclareMathOperator{\Coeff}{Coeff}

\nc{\myi}{{\rm{i}}}
\nc\UU{\mbox{\bfseries U}}
\nc\FF{\mbox{\bfseries \itshape F}}
\nc\h{\mbox{\bfseries \itshape h}}\nc\dd{\mbox{d}}
\nc\g{\mbox{\bfseries \itshape g}}
\nc\xx{\mbox{\bfseries \itshape x}}

\def\N{\mathbb{N}}
\def\Z{\mathbb{Z}}
\def\Q{\mathbb{Q}}

\def\CP{\mathbb{CP}}

\def\ze{\zeta}

\def\xx{\left(\frac{1-x}{1+x} \right)}

\def\ol{\overline}
\nc{\olT}{{\ol{T}}}
\nc\divg{{\text{div}}}
\theoremstyle{plain}
\newtheorem{thm}{Theorem}[section]
\newtheorem{lem}[thm]{Lemma}

\newtheorem{cor}[thm]{Corollary}
\newtheorem{conj}[thm]{Conjecture}

\theoremstyle{definition}

\newtheorem{rem}[thm]{Remark}
\newtheorem{eg}[thm]{Example}

\nc{\cicc}[1]{{}_{{}^{ \bigcirc\hskip-1.2ex{#1}\hskip.3ex{}}}}
\nc{\cic}[1]{{}^{\bigcirc\hskip-1.15ex{\raisebox{-0.015cm}{\text{$\scriptscriptstyle #1$}}}\hskip.25ex{}}}
\nc{\ccic}[1]{{}^{\bigcirc\hskip-1.5ex{\raisebox{-0.015cm}{\text{$\scriptscriptstyle #1$}}}\hskip.25ex{}}}
\nc{\ncic}[1]{ {\bigcirc\hskip-1.6ex{\raisebox{-0.0cm}{\text{$\scriptstyle #1$}}}\hskip.25ex{}}}
\nc{\nncic}[1]{ {\bigcirc\hskip-2ex{\raisebox{-0.0cm}{\text{$\scriptstyle #1$}}}\hskip.25ex{}}}
\nc{\cci}[1]{{}_{{}^{ {\textstyle \bigcirc}\hskip-2.05ex{#1}\hskip-.35ex{}}}}
\nc{\ccicc}[1]{{}_{{}^{ {\textstyle \bigcirc}\hskip-1.55ex{#1}\hskip-0.1ex{}}}}
\nc{\x}{\rm{x}}
\nc{\tworow}[2]{\left(#1 \atop #2\right)}

\nc{\fl}{{\mathfrak l}}
\nc{\fm}{{\mathfrak m}}

\begin{document}
%%%%%%%%%%%%%%%%%%%% title %%%%%%%%%%%%%%%%%%%%%%%%%%%%%%%%%%%%%%%%%%%%%%%%
\title{\bf Unramified Motivic Alternating Multiple Mixed Values}
\author{
{Ce Xu${}^{a,}$\thanks{Email: cexu2020@ahnu.edu.cn}\ \ and Jianqiang Zhao${}^{b,}$\thanks{Email: zhaoj@ihes.fr}}\\[1mm]
\small a. School of Mathematics and Statistics, Anhui Normal University, Wuhu 241002, PRC\\
\small b. Department of Mathematics, The Bishop's School, La Jolla, CA 92037, USA
%\\[5mm]\emph{\normalsize Dedicated to Professor Masanobu Kaneko on the occasion of his 65th birthday}
}
\date{}
\maketitle

\noindent{\bf Abstract.} Many variants of the multiple zeta values have been studied in recent years. There are a few among them that remain to be real numbers, such as the multiple mixed values defined by the authors as level-two generalizations, which include both Hoffman's multiple $t$-values and Kaneko--Tsumura's multiple $T$-values. A central question is when such a value descends to level one, that is, when it can be expressed as a $\Q$-linear combination of multiple zeta values. We call these values unramified. In this paper, we further consider the alternating version of the above variants and identify five families of unramified alternating multiple mixed values using the descent theory of Brown and Glanois. We conjecture that all unramified truly alternating multiple mixed values are given in this paper.

\medskip


\medskip


%\setcounter{section}{-1}
%\tableofcontents

\section{Introduction}
Multiple zeta values (MZVs), defined for positive integers $\bfs=(s_1, \dots, s_d)\in\N^d$ with $s_d \ge 2$ by the nested sums
\begin{equation*}
\zeta(\bfs) = \sum_{0<n_1< \cdots < n_d} \frac{1}{n_1^{s_1} \cdots n_d^{s_d}},
\end{equation*}
have been a central object of study in number theory, geometry, and mathematical physics for for roughly three decades. Their algebraic structure, governed by the celebrated (regularized) double shuffle relations, and their deep connections to mixed Tate motives (due to Deligne \cite{Deligne1989,Deligne2010}, Goncharov\cite{DeligneGo2005}, and many others) have cemented their status as a central class of periods. The interplay between their analytic properties and motivic incarnations has led to profound results, including the partial resolution of the Hoffman conjecture by Brown \cite{Brown2012} and the construction of the Grothendieck-Teichm\"uller group via their combinatorics.

A natural and fruitful direction of inquiry is the study of generalizations of MZVs. Among these, level two variants (i.e. the alternating MZVs and related variations) have attracted particular attention due to their rich arithmetic structure and close ties to many other branches of mathematics. Notable examples include Hoffman's multiple $t$-values \cite{Hoffman2019}, Kaneko--Tsumura's multiple $T$-values \cite{KanekoTs2020}, and the authors' multiple $S$-values \cite{XuZhao2020a}. These objects often exhibit surprising and deep relationships with their level-one counterparts (i.e., classical MZVs). Furthermore, the authors introduced the broader class of ``multiple mixed values'' (MMVs) in \cite{XuZhao2020a}, a comprehensive unification of these level two objects, providing a systematic framework for studying their algebra and geometry. They further extend the setting to alternating MMVs in \cite{XuYanZhao2022Aug}, with elements living in level four, which are the main objects of study of this paper.

Given this hierarchy of levels, a central and fundamental question arises: When does a level-two or level four object descend to level one? That is, under what conditions can a given (alternating) MMV be expressed as a $\Q$-linear combination of classical MZVs? In \cite{XuZhao2026July1}, we gave a complete motivic answer for all MMVs with depth less than four and formulated several conjectures describing the general situation.

The above question goes back to Broadhurst \cite{Broadhurst1996a} who called such elements ``honorary MZVs''. Following the terminology established in the motivic framework of Brown and Glanois, we refer to them as ``unramified''. This question is not merely a formal curiosity; it is closely tied to the motivic Galois theory of MZVs. Unramified elements at level four correspond to periods that survive the projection from the level-four motivic fundamental group to its abelianization, or more concretely, they represent classes that lie in the image of the inclusion of the level-one motivic Lie algebra into the level-four one. Characterizing these unramified classes is equivalent to understanding the kernel of the natural surjection from the level-four motivic Galois group to the level-one group, a problem intimately related to the action of the Grothendieck-Teichm\"uller group on unipotent fundamental groups of the thrice punctured complex plane $\CP^1\setminus \{0, 1, \infty\}$.

A powerful framework for addressing these problems is the descent theory developed systematically by Brown \cite{Brown2013a} and refined by Glanois \cite{Glanois2015}. Originally formulated to study the distribution of motivic MZVs and their Galois actions, descent theory provides a cohomological and combinatorial toolkit to lift identities from the level of formal symbols (the associated algebras) back to the level of actual periods. By analyzing the action of the descent operators, one can determine precisely which algebraic combinations of higher-level MZVs vanish upon applying the ``unramified'' projection.

In this paper, we apply this descent-theoretic approach to the alternating MMVs. Our main results are threefold. First, we derive the explicit criterion theorems (see Theorem~\ref{thm:criterion1} and  Theorem~\ref{thm:criterion}) for a motivic color/cyclotomic MZV at level four to be unramified by Brown--Glanois descent theory. Second, we apply this criterion to two classes of motivic alternating MMVs by computing the coactions and show they are unramified. Then by applying duality we can uncover three more unramified families. Third, based on the structural patterns observed in our computations, we propose a conjecture giving a complete characterization of the unramified elements among all the truly alternating MMVs (i.e., those that are not MMVs).

The paper is organized as follows. In Section \ref{sec:setup}, we review the necessary background on the Brown-Glanois descent theory tailored to our level four setting. In Section \ref{sec:MMVsDefn}, we recall the definitions of MMVs and their alternating extensions, and provide their motivic definition based on their iterated integral forms. In Sections~\ref{sec:DblAltV} and~\ref{sec:TripleAltV}, we carry out the explicit descent computations and prove the unramifiedness of the two families of alternating MMVs. Section~\ref{sec:moreAltV} presents three more such unramified families via duality relations. At the end of the paper, we state our general conjecture which essentially says that no other unramified truly alternating MMVs exist other than the ones we have found in this paper.

\section{Motivic setup and descent criterion theorems at level four}\label{sec:setup}

\subsection{The motivic setup in level four}
Our main objects of study in this paper live in the level-four part of the so-called colored/cyclotomic MZVs.
For any \emph{level} $N\in\N$, let $\gG_N$ denote the group of $N$-th roots of unity. For  $\bfk=(k_1,\dots,k_d)\in\N^d, \bfz=(z_1,\dots,z_d)\in\gG_N^d$, the colored/cyclotomic MZV (or CMZV) is defined by the series
\begin{equation*}
\Li_\bfk(\bfz):=\zeta\binom{\bfk}{\bfz}:
= \sum_{0 < n_1 < \cdots < n_d} \frac{z_1^{n_1} \cdots z_d^{n_d}}{n_1^{k_1} \cdots n_d^{k_d}},
\end{equation*}
with the condition $(k_d,z_d) \neq (1,1)$ to ensure convergence; we call $(\bfk;\bfz)$ \emph{admissible}. The quantity $|\bfk|:= k_1 + \cdots + k_d$ is called the \emph{weight}, and $d$ the \emph{depth}. They are equipped with an iterated integral representation:
\begin{align*}
\Li_\bfk(\bfz) =&\, I(0;\eta_1,0_{k_1-1},\eta_2,0_{k_2-1},\ldots,\eta_d,0_{k_d-1};1)\\
:=&\,   \int_0^1 \tx_{\eta_1} \tx_0^{k_1-1} \dots  \tx_{\eta_d} \tx_0^{k_d-1}
\end{align*}
where $\eta_j=1/(z_jz_{j+1}\cdots z_d)$ for all $j\ge 1$, $\tx_0=dt/t$, and $\tx_\ga=dt/(\ga-t)$ for all $\ga\in\gG_N$.

A fruitful recent approach to studying CMZVs is to consider their motivic versions. These are motivic periods of the fundamental groupoid of $\PP^1 \setminus \{0, \gG_N, \infty\}$. Under Grothendieck's period conjecture, the period map from motivic periods to complex periods is an isomorphism; thus, identities established at the motivic level descend to identities among ordinary CMZVs.

Brown developed the theory of motivic MZVs and used it to prove Hoffman's conjecture on the generators of the space of MZVs. Glanois extended this framework to the cyclotomic setting, introducing definitions of motivic cyclotomic MZVs for $\gG_N$ with $N=2,3,4,6,8$, and provided bases in each case. In what follows, we will focus on the cases when $N=1,2,4$.

Let $w\in\N$ and suppose $a_j=0$ or $a_j\in\Gamma_{N}$ for all $0\le j\le w+1$. Then we can define the motivic integrals
$I^\fm(a_0;a_1,\dots,a_w;a_{w+1})$ subject to the following axioms given below. Moreover, there is a period map $\dch$ such that
\begin{equation}\label{equ:periodMap}
\dch \big(I^\fm(a_0;a_1,\dots,a_w;a_{w+1})\big)=  \int_{a_0}^{a_{w+1}}\tx_{a_1} \cdots \tx_{a_w}
\end{equation}
as an iterated integral whenever it converges (to some CMZV). If it does not converge, then the right-hand side must be replaced by its regularization values (more precisely, its shuffle regularized polynomial evaluated at 0, see \cite[Ch.\ 14]{Zhao2016} for more details).

Let $\calH^N$ be the $\Q$-spann of the \emph{motivic CMZVs} of the form
\begin{equation*}
\zeta^\fm_a\binom{k_1,\dots,k_d}{z_1,\dots,z_d}=\zeta^\fm_a(k_1,\dots,k_d;z_1,\dots,z_d):= I^\fm(0;0_a,\eta_1,0_{k_1-1},\dots,\eta_d,0_{k_d-1};1),
\end{equation*}
where $a\in\N_0$, $k_j\in\N, z_j\in\Gamma_{N}$, $\eta_j=1/z_j\cdots z_d$ for all $j=1,\dots,d$.
We simply write $\zeta^\fm$ for $\zeta^\fm_0$.

We now recall the relevant properties and axioms of the motivic integrals as follows (cf. \cite[p.~9]{Glanois2016} and \cite[\S2, (I1)-(I6)]{Murakami2021}):
\begin{itemize}
	\item[(I1)] Empty word: $I^\fm(a_{0}; a_{1})=1$.

	\item[(I2)] Trivial path: $\forall$ weight $n\ge1$, $I^\fm(a_{0}; a_{1}, \dots, a_{n}; a_{n+1})=0$ if $a_{0}=a_{n+1}$.

	\item[(I3)] Shuffle product: $\forall a\in\N_0$ we have
	\begin{equation*}
\zeta_a^\fm \binom{k_1,\dots,k_d}{z_1,\dots,z_d}=
(-1)^a\sum_{\substack{i_{1}+ \cdots + i_d=a\\ i_{1},\dots,i_d\ge0}} \left(\prod_{j=1}^d\binom {k_j+i_j-1} {i_j}\right) \zeta^\fm \binom{k_1+i_1, \cdots , k_d+i_d}{z_1\ \ , \cdots ,\ \ z_d}.
 \end{equation*}

	\item[(I4)] Regularization: If $a_{1}=\cdots=a_n\in \{0,1\}$, then
	$I^\fm(0; a_{1}, \cdots, a_{n}; 1)=0.$

	\item[(I5)] Path reversal: $I^\fm(a_{0}; a_{1}, \cdots, a_{n}; a_{n+1})= (-1)^n I^\fm(a_{n+1}; a_{n}, \cdots, a_{1}; a_{0}).$

	\item[(I6)] Homothety: $\forall \alpha \in \Gamma_{N}, I^\fm(0; \alpha a_{1}, \cdots, \alpha a_{n}; \alpha a_{n+1}) = I^\fm(0; a_{1}, \cdots, a_{n}; a_{n+1})$.

	\item[(I7)] Change of variable $t\to 1-t: \forall a_{1}, \cdots, a_{n}\in\{0, 1\},$ $$I^\fm(0; a_{1}, \cdots, a_{n}; 1)= I^\fm(0;1-a_n, \cdots, 1-a_1; 1).$$

	\item[(I8)] Path composition: $\forall a,b, x\in \Gamma_{N} \cup \left\{0\right\}$,
	$$ I^\fm(a; a_{1}, \cdots, a_{n}; b)=\sum_{i=0}^{n} I^\fm(a; a_{1}, \cdots, a_{i}; x) I^\fm(x; a_{i+1}, \cdots, a_{n}; b) .$$
\end{itemize}

In what follows, we often suppress the alternating signs by putting a bar on top of $k_j$ if the correspond $z_j=-1$. For example, $\zeta^\fm_3(\ol{4},5)=\zeta^\fm_3(4,5;-1,1)$.

\subsection{The descent criterion for CMZVs of level four}
The motivic CMZVs carry a  Hopf comodule structure, dual to the action of the motivic Galois group on these motivic periods. The coaction is given explicitly by a combinatorial formula due to Goncharov and Brown. This structure allows one to study the Galois-theoretic properties of CMZVs purely algebraically.

The motivic Galois group of mixed Tate motives over $\Z[\gG_N, 1/N]$ acts faithfully on the motivic fundamental group. By analyzing this action, one can compute coefficients and derive relations among CMZVs. This approach has been particularly successful in understanding unit CMZVs and in proving explicit relations in substantial cases (see, e.g., \cite{Charlton-MathAnn2025,Keilthy2022,Li2024,Murakami2021}).

We now briefly review this setup in some detail. Put $\calA^2=\calH^2/\zeta^\fm(2)\calH^2$ and $\calA^4=\calH^4/(2\pi \myi)^\fm\calH^4$
by killing the Lefschetz motive $(2\pi \myi)^\fm:={\mathbb L}=\Q(-1)$, which has period $2\pi \myi$. Denote by
$\calH^N_w$ and $\calA_w^N$ their weight $w$ parts for all $w\ge 1$. Let $\calL^N=\calA^N_{>0}/\calA^N_{>0}\cdot \calA^N_{>0}$.
For any weight $w$ and integer $r$ such that $0<r<w$ by modifying the coproduct of a suitable Hopf algebra
one can define a derivation as part of a coaction
$$
D_r: \calH_w^N \to \calL_r^N \ot \calH_{w-r}^N
$$
by sending $I^\fm(a_0;a_1,\dots,a_w;a_{w+1})$ to
$$
\sum_{p=0}^{w-r} I^\fl(a_p; a_{p+1},\dots,a_{p+r};a_{p+r+1})\ot I^\fm(a_0;a_1,\dots,a_p,a_{p+r+1},\dots,a_w;a_{w+1}).
$$
Each summand in the above sum is called a \emph{cut}.

In order to state the descent criterion at level four, we need the following lemma.

\begin{lem}\label{lem:calL_r^4}
The weight $1$ piece of $\calL^4$ is
\begin{equation*}
\calL_1^4=\langle \Li^\fl_1(-1)=2\Li^\fl_1(\myi)=2\Li^\fl_1(-\myi)\rangle.
\end{equation*}
If $r>1$ then
\begin{equation*}
\calL_r^4=
\left\{
  \begin{array}{ll}
    \langle\zeta^\fl(r)\rangle, & \hbox{if $r$ is odd;} \\
    \langle\Li^\fl_r(\myi)\rangle, & \hbox{if $r$ is even.}
  \end{array}
\right.
\end{equation*}
\end{lem}
\begin{proof}
From the table of values in \cite[\S5.2.1]{Glanois2015}, we see that $\Li^\fl_1(\myi)=\Li^\fl_1(-\myi)$. In fact,
$\Li_1(\myi)$ and $\Li_1(-\myi)$ have the same real part but opposite imaginary parts $\pm \pi \myi/2$. But the
Lefschetz motive $(2\pi \myi)^\fm$ vanishes after passing to the de Rham quotient in $\calA^4$:
\begin{equation}\label{equ:2pifl=0}
   (2\pi \myi)^\fl=0.
\end{equation}

Further, we know that $\Li^\fl_1(1)=I^\fl(0;1;1)=0$ and
$\Li^\fl_1(-1)=I^\fl(0;-1;1)=-\log^\fl 2=2\Li^\fl_1(\myi)=2\Li^\fl_1(-\myi)$.

For $r>1$, from loc. cit. we have
the relation
\begin{equation}\label{equ:calL_r}
\Li^\fl_r(\myi)=
\left\{
  \begin{array}{ll}
    -\Li^\fl_r(-\myi), & \hbox{if $2|r$;} \\
    \Li^\fl_r(-\myi)=2^{-r} \zeta^\fl(\ol{r})=2^{-r}(2^{1-r}-1)\zeta^\fl(r)\in\calL_r, & \hbox{if $2\nmid r>1$.}
  \end{array}
\right.
\end{equation}
The claim follows from the fact that $\zeta^\fl(r)=0$ for all even $r$.
\end{proof}

\begin{cor}\label{cor:tr}
If $r\in\N$ is odd  then we have $t^\fl(\ol{r})=0$, where the motivic MtVs are defined by \eqref{equ:tfm}.
\end{cor}
\begin{proof}
This follows immediately from the lemma, since $t^\fl(\ol{r})=\myi\big( \Li^\fl_r(\myi)-\Li^\fl_r(-\myi)\big)$.
\end{proof}

Set $\calH=\calH^1$ and $\calL=\calL^1$. For $r=1$ or even $r>1$, define
$D_r^\myi: \calH_w^4 \to \calH_{w-r}^4$
to be the composition $\pi_r^\myi\circ D_r$ where
$\pi_r^\myi: \calL_r^4 \ot \calH_{w-r}^4 \to \calH_{w-r}^4$
is defined by
\begin{equation*}
 \pi_r^\myi \Big(\ga  \Li^\fl_r(\myi) \ot h\Big)
 = \ga   h \qquad \forall h\in \calH_{w-r}^4
\end{equation*}
and then extended to $\calL_r^4\ot \calH_{w-r}^4$ by $\Q$-linearity.

We are now ready to state the key criterion theorem for unramified CMZVs of level four.
\begin{thm}\label{thm:criterion1}
Let $w\in\N$ and let $h\in\calH_w^4$. Then $h\in\calH_w$
if and only if the following three conditions are all satisfied:
\begin{enumerate}
  \item $D_r^\myi h=0$ for $r=1$ and all even $r<w$,
  \item $D_r h \in \calL_r\ot \calH_{w-r}$ for all odd $r<w$.
\end{enumerate}
Furthermore, $h\in \zeta^\fm(w)\Q$ if and only if $D_r h=0$ for all positive integers $r<w$.
\end{thm}
\begin{proof}
We need to check that the necessary and sufficient conditions in the theorem are exactly the same ones as in
\cite[Theorem\ 5.1.1]{Glanois2015} for level $N=4$. Indeed, when $r=1$, there are one new descent derivation
$\calD^\lceil_1=\{D_1^\myi\}$ while $\calD^{\setminus\lceil}_1=\emptyset$ corresponding to the fact that
$\calL_1=\emptyset$ and $\calL_1^4=\langle \Li^\fl_1(\myi)\rangle$ has dimension 1 by Lemma~\ref{lem:calL_r^4}.

Similarly, if $r>1$ is odd then $\calL_r = \langle\zeta^\fl(r)\rangle=\calL_r^4$ which implies that
$\calD^\lceil_r=\emptyset$ and $\calD^{\setminus\lceil}_r=\{D_r\}$ having only one derivation\footnote{Notice that we don't need to specify roots of unity for $D_r$ here as there is only one root in level one, namely, we always have $D_r=D_r^1$.}.
If $r$ is even then $\calL_r =\emptyset$ and $\calD^{\setminus\lceil}_r=\emptyset$ while $\calL_r^4=\langle\Li^\fl_r(\myi)\rangle$
which means that $\calD^\lceil_r=\{D_r^\myi\}.$

Therefore, our theorem says that $h\in\calH_w^4$ is unramified (i.e., $h\in\calH_w^1$) if and only if $D_r^\myi(h)=0$ for all $D_r^\myi\in \calD^\lceil$ ($\ne \emptyset$ if and only $r=1$ and $r$ is even) and $D_r(h)\in \calL_r\ot \calH_{w-r}$ for all $D_r\in\calD^{\setminus\lceil}_r$ ($\ne \emptyset$ if and only $1<r$ is odd), which is exactly the conditions
in \cite[Theorem\ 5.1.1]{Glanois2015}.

The last sentence of the theorem follows from \cite[Theorem\ 3.3]{Brown2012}.
\end{proof}

The following real version of Theorem~\ref{thm:criterion1} is its easy corollary because of the simple structure of $\calL_r^4$ provided by Lemma~\ref{lem:calL_r^4}.
\begin{thm}\label{thm:criterion}
Let $w\in\N$ and suppose $h\in\calH_w^4$ is invariant under conjugation (namely, the real Frobenius). Then $h\in\calH_w$
if and only if $D_1 h=0$ and $D_r h \in \calL_r\ot \calH_{w-r}$ for all odd $r<w$. Additionally, $h\in \zeta^\fm(w)\Q$ if and only if $D_r h=0$ for all positive integers $r<w$.
\end{thm}
\begin{proof}
Let even $r<w$ and suppose
\begin{equation*}
 D_r h= c \Li^\fl_r(\myi)\ot R^\fm_{w-r} \in \calL_r^4 \ot H_{w-r}^4
\end{equation*}
for some $c\in\Q$ and $R^\fm_{w-r}\in H_{w-r}^4$. Since conjugation commutes with the derivation, we must have
\begin{equation*}
c \Li^\fl_r(\myi)=c \Li^\fl_r(\ol{\myi})=c \Li^\fl_r(-\myi)=-c\Li^\fl_r(\myi)
\end{equation*}
by \eqref{equ:calL_r}.
Therefore,  $c=0$ and the theorem follows from Theorem~\ref{thm:criterion1} immediately.
\end{proof}

\section{Alternating multiple mixed values}\label{sec:MMVsDefn}
For $k_1,\ldots,k_d\in \N$ and $\eps_1,\ldots,\eps_d\in\{\pm 1\}$ with $(k_d,\sigma_d)\ne (1,1)$, the \emph{alternating multiple mixed values} (MMVs) are defined by (\cite{XuYanZhao2022Aug})
\begin{align*}
&M(k_1,\ldots,k_d;\eps_1,\ldots,\eps_d;\sigma_1,\ldots,\sigma_d)\nonumber\\
&:=\sum_{0<m_1<\cdots<m_d} \frac{(1+\eps_1 (-1)^{m_1})\sigma_1^{({2m_1+1-\eps_1})/{4}}\cdots (1+\eps_d (-1)^{m_d})\sigma_d^{({2m_d+1-\eps_d})/{4}}}{m_1^{k_1}\cdots m_d^{k_d}}.
\end{align*}
A direct calculation yields the following iterated integral form:
\begin{align}\label{defn-AMMV-integrals}
&M(k_1,\ldots,k_d;\eps_1,\ldots,\eps_d;\sigma_1,\ldots,\sigma_d)\nonumber\\
&=\sum_{\eta_j\in\{\pm 1\}} \left(\prod_{j=1}^d (\theta_j\eta_j)^{(\eps_j-\eps_{j-1})/2}\right) I(0;\theta_1 \eta_1,0_{k_1-1},\theta_2 \eta_2,0_{k_2-1},\ldots,\theta_d \eta_d,0_{k_d-1};1),
\end{align}
where $\eps_0:=1$ and $\theta_j:=(\sigma_j\cdots \sigma_d)^{-1/2}$. Here and throughout the paper, we define $1^{1/2}=1$ and $(-1)^{1/2}=\sqrt{-1}=\myi$, and $\eta_j\in\{\pm 1\}$ means $\eta_1,\dots,\eta_d\in\{\pm 1\}$.

As a special case, taking $\eps_1=\cdots=\eps_d=-1$ above, we obtain the \emph{alternating multiple $t$-values} (MtVs)
\begin{align*}
t(k_1,\ldots,k_d;\sigma_1,\ldots,\sigma_d)&:=\sum_{0<n_1<\cdots<n_d} \frac{2^d\sigma^{n_1}_1\cdots \sigma_d^{n_d}}{(2n_1-1)^{k_1}\cdots (2n_d-1)^{k_d}}\\
&=\sqrt{\sigma_1\cdots \sigma_d}\sum_{\xi_j\in\{\pm 1\},\atop j=1,2,\ldots,d } \xi_1\cdots\xi_d \Li_{k_1,\ldots,k_d}\bigg(\xi_1\sqrt{\sigma_1},\ldots,\xi_d\sqrt{\sigma_d}\bigg)\\
&=\theta_1^{-1}\sum_{\eta_{j}\in\{\pm 1\},\atop j=1,2,\ldots,d} \eta_1 I(0;\eta_1\theta_1,0_{k_1-1},\eta_2\theta_2,0_{k_2-1},\ldots,\eta_d\theta_d,0_{k_d-1};1),
\end{align*}
where $\eta_j:=\xi_j\cdots \xi_d$ and $\theta_j:=(\sigma_j\cdots \sigma_d)^{-1/2}$. Furthermore, when all $\sigma_j = 1$, the classical MtVs $t(k_1,\ldots,k_r)$ are recovered, which were first systematically introduced and studied by Hoffman \cite{Hoffman2019}. Subsequently, Murakami \cite{Murakami2021} and Charlton \cite{Charlton-MathAnn2025} further developed the motivic framework, lifting the classical MtVs to motivic MtVs, and thereby uncovering numerous algebraic relations among them.

Motivated by \eqref{defn-AMMV-integrals}, we define the motivic alternating MMVs by
\begin{align}\label{defn-AMMV-Motivic}
&M_a^\fm(k_1,\ldots,k_d;\eps_1,\ldots,\eps_d;\sigma_1,\ldots,\sigma_d)\nonumber\\
&=\sum_{\eta_j\in\{\pm 1\}} \left(\prod_{j=1}^d (\theta_j\eta_j)^{(\eps_j-\eps_{j-1})/2}\right) I^\fm(0;0_a,\theta_1 \eta_1,0_{k_1-1},\theta_2 \eta_2,0_{k_2-1},\ldots,\theta_d \eta_d,0_{k_d-1};1),
\end{align}
for all $a\in\N_0$, and similarly the motivic alternating MtVs by
\begin{align}\label{equ:tfm}
t_a^\fm(k_1,\ldots,k_d;\sigma_1,\ldots,\sigma_d):=\theta_1^{-1}\sum_{\eta_{j}\in\{\pm 1\},\atop j=1,2,\ldots,d} \eta_1 I_a^\fm(0;\eta_1\theta_1,0_{k_1-1},\eta_2\theta_2,0_{k_2-1},\ldots,\eta_d\theta_d,0_{k_d-1};1).
\end{align}

\section{A family of motivic alternating double $t$-values}\label{sec:DblAltV}
We now apply the criterion Theorem~\ref{thm:criterion} to prove the following result. In what follows, we often suppress the alternating signs of alternating MMVs by putting a bar on top of $k_j$ if the corresponding $\sigma_j=-1$.

\begin{thm}\label{thmone-doublealtert}
For all $a,b\in\N$, the motivic alternating double $t$-value $t^\fm(\ol{2a+1},\ol{b+1})$ is unramified.
\end{thm}

\begin{proof}
Observing that $t^\fm(\ol{2a+1},\ol{b+1})$ is real, we can apply Theorem~\ref{thm:criterion} by
utilizing the descent criterion and computing the coaction images under $D_r$ for all odd $r<w:=2a+b+2$.
Hence, it suffices to consider the following cuts.
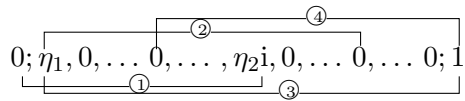
\begin{center}
\begin{tikzpicture}[scale=0.9]
\node (A0) at (0.05,0) {$0;\eta_1,0,\ldots\, 0, \ldots\, ,\eta_2 \myi, 0, \ldots\,0, \ldots\, 0;1$};
\node (A1) at (-1.4,-0.4) {${}$};
\node (A2) at (-0.5,0.4) {${}$};
\node (A3) at (0.8,-0.5) {${}$};
\node (A4) at (1.2,0.6) {${}$};
\draw (-3.12,-0.25) to (-3.12,-0.4) to (A1) node {$\cic{1}$} to (0.4,-0.4) to (0.4,-0.25);
\draw (-2.8,0.15) to (-2.8,0.4) to (A2) node {$\cic{2}$}to  (1.85,0.4) to (1.85,0.20);
\draw (-2.8,-0.25) to (-2.8,-0.5)   to (A3) node {$\cic{3}$}to (3.3,-0.5) to (3.3,-0.25);
\draw (-1.15,0.15) to (-1.15,0.6)   to (A4) node {$\cic{4}$}to (3.3,0.6) to (3.3,0.2);
\node (A) at (0,-1.6) {Possible cuts of $D_r t^\fm(\ol{2a+1},\ol{b+1})$ for all $1\le r\le 2a+b+1$.};
\end{tikzpicture}
\end{center}
Then
\begin{align*}
&\, D_r t^\fm (\ol{2a+1},\ol{b+1})=D_r \sum_{\eta_j\in\{\pm 1\}} \eta_1 I^\fm(0;\eta_1,0_{2a},\eta_2 \myi,0_{b};1) =\ncic{1}+\cdots+\ncic{4}
\end{align*}
where
\begin{align*}
\ncic{1}=&\, \delta_{r=2a+1} \sum  \eta_1 I^\fl(0;\eta_1,0_{2a};\eta_2 \myi)  \ot I^\fm(0;\eta_2 \myi,0_{b};1) \\
=&\, \delta_{r=2a+1} \sum  \eta_1 I^\fl(0;-\eta_1\myi,0_{2a};1)  \ot \eta_2 I^\fm(0;\eta_2 \myi,0_{b};1) \\
=&\, \delta_{r=2a+1} t^\fl(\ol{2a+1})  \ot t^\fm(\ol{b+1})=0 \quad\text{(Cor.~\ref{cor:tr}),}\\
\ncic{2}=&\,\delta_{r\ge 2a+1} \sum  \eta_1 I^\fl(\eta_1;0_{2a},\eta_2 \myi,0_{r-2a-1};0)  \ot I^\fm(0;\eta_1,0_{2a+b-r+1};1) \\
=&\, -\delta_{r\ge 2a+1} \sum I^\fl(0;0_{r-2a-1},\eta_2 \myi,0_{2a};1)  \ot \eta_1 I^\fm(0;\eta_1,0_{2a+b-r+1};1) \\
=&\, -\delta_{r\ge 2a+1}\bigg[ \ze^\fl_{r-2a-1}\binom{2a+1}{\myi}+\ze^\fl_{r-2a-1}\binom{2a+1}{-\myi}\bigg]  \ot t^\fm(2a+b-r+2) \\
=&\, \delta_{r\ge 2a+1}(-1)^r \binom{r-1}{2a}\big(\Li^\fl_r(\myi)+ \Li^\fl_r(-\myi)\big)   \ot t^\fm(2a+b-r+2)\\
\in &\,    \calL_r\ot\calH_{w-r}, \\
\ncic{3}=&\,\delta_{r=2a+b+1} \sum  \eta_1 I^\fl(\eta_1;0_{2a},\eta_2 \myi,0_{r-2a-1};1)  \ot I^\fm(0;\eta_1;1) \\
=&\,\delta_{r=2a+b+1} \sum \bigg( \eta_1 I^\fl(0;0_{2a},\eta_2 \myi,0_{r-2a-1};1)  \ot I^\fm(0;\eta_1;1) \\
&\, +\eta_1 I^\fl(\eta_1;0_{2a},\eta_2 \myi,0_{r-2a-1};0)  \ot I^\fm(0;\eta_1;1) \bigg)=0 \quad\text{(by path reversal),}\\
\ncic{4}=&\,\delta_{r\ge b+1} \sum  \eta_1 I^\fl(0;0_{r-b-1},\eta_2 \myi,0_b;1)  \ot I^\fm(0;\eta_1,0_{2a+b-r+1};1) \\
=&\,\delta_{r\ge b+1} \bigg[ \ze^\fl_{r-b-1}\binom{b+1}{\myi}+ \ze^\fl_{r-b-1}\binom{b+1}{-\myi}\bigg]  \ot t^\fm(2a+b-r+1) \\
=&\,\delta_{r\ge b+1}(-1)^{r-b-1} \binom{r-1}{b} \big(\Li^\fl_r(\myi)+ \Li^\fl_r(-\myi)\big)  \ot t^\fm(2a+b-r+1) \\
\in &\,    \calL_r\ot\calH_{w-r}
\end{align*}
since $t^\fm(s)$ is unramified if $s\ge 2$ according to \cite[Theorem~1]{Murakami2021}. The computation above immediately implies that $t^\fm(\ol{2a+1},\ol{b+1})$ is unramified by Theorem~\ref{thm:criterion}.
\end{proof}

\begin{cor}
Let $m\in\N_{>1}$ be odd, $n\in\N$ be even, and let $w=m+n$. Then we have
\begin{align}\label{equ:dblAltWtodd}
 t^\fm(\ol{m},\ol{n})=-t^\fm(w)+ \sum_{3\le r<w, \ 2\nmid r}\bigg[\binom{r-1}{m-1}+\binom{r-1}{n-1}\bigg]  \frac{(2^{r-1}-1)(2^{w-r}-1)}{2^{w+r-3}}\zeta^\fm(r)\zeta^\fm(w-r).
\end{align}
\end{cor}
\begin{proof}
We first need to show that for all odd $r<w$,
\begin{equation*}
D_r  t^\fm(\ol{m},\ol{n}) = \bigg[\binom{r-1}{m-1}+\binom{r-1}{n-1}\bigg]  \frac{(2^{r-1}-1)(2^{w-r}-1)}{2^{w+r-3}}\zeta^\fl(r)\ot \zeta^\fm(w-r).
\end{equation*}
But this follows from the computation in the proof of Theorem~\ref{thmone-doublealtert} after simplification using \eqref{equ:calL_r}.
Moreover, setting $\sigma_1 = \sigma_2 = -1$ with $p$ odd and $q$ even in \cite[Theorem~10]{Xu-Wang2023} (also see \cite{Quan2020,XuYan2026}), we obtain the analytic version of \eqref{equ:dblAltWtodd}. The corollary then follows by Theorem~\ref{thm:criterion}.
\end{proof}

\begin{eg}
For instance,
\begin{align*}
&t(\bar{3},\bar{2}) = -\frac{31}{16}\ze(5)+\frac{27}{32}\ze(2)\ze(3),\\
&t(\bar{3},\bar{4})=-\frac{127}{64}\ze(7)+\frac{225}{256}\ze(2)\ze(5)+\frac{45}{128}\ze(3)\ze(4).
\end{align*}
\end{eg}

\begin{rem}
We also attempted and failed to find a formula expressing $t(\overline{2a+1},\overline{2b+1})$ in terms of products of Riemann zeta values.
We know that Glanois's motivic Galois descent criterion in \cite{Glanois2015} cannot specify the depth of the descent, as found in \cite{Charlton2026}. The motivic coaction does not allow one to fix the depth of classical MZVs, as the coradical filtration and the depth filtration do not agree. This highlights the need for more refined tools to address depth-related questions. Moreover, numerical evidence suggest that MZVs of depth two and four are required to expressing $t(\overline{2a+1},\overline{2b+1})$ in general.
\end{rem}

\section{A family of motivic alternating triple $t$-values}\label{sec:TripleAltV}
In this section, we present another family of unramified alternating MtVs. It is closely related to the family in Section~\ref{sec:DblAltV}. Recall that we have a variant form of motivic alternating MZVs: for $\bfk=(k_1,\dots,k_d)\in\N^d$ and $\bfgs=(\sigma_1,\dots,\sigma_d)\in\{\pm1\}^d$,
\begin{align*}
\tz^\fm_a(\bfk;\bfgs):=&\, \sum_{\eta_j=\pm 1} I^\fm(0;0_a,\eta_1\theta_1,0_{k_1-1},\dots,\eta_d\theta_d,0_{k_d-1};1),
\end{align*}
where $\theta_j:=(\sigma_j\cdots \sigma_d)^{-1/2}$. When $a=0$ and $(\bfk;\bfgs)$ is admissible we can see easily that
\begin{align*}
\dch \Big(\tz^\fm(\bfk;\bfgs)\Big):=&\, 2^{d-|\bfk|} \zeta(\bfk;\bfgs).
\end{align*}

\begin{thm}\label{thm:tripleAltMtVs}
All alternating triple $t$-values $t^\fm(\ol{2a+1},\ol{2b},3)\ (a,b\in\N)$ are unramified.
\end{thm}

\begin{proof}
Let $w=2a+2b+4$. We first notice that
\begin{equation*}
t^\fm(\ol{2a+1},\ol{2b},3)=\sum_{\eta_j\in\{\pm 1\}} \eta_1 I^\fm(0;\eta_1,0_{2a},\eta_2 \myi,0_{2b-1},\eta_3,0,0;1)
\end{equation*}
is invariant under conjugation. Thus, by Theorem~\ref{thm:criterion} we can safely disregard the cuts corresponding to
\begin{equation*}
\ncic{4}: r=2a+2b, \quad \ncic{6}: r=2a+2b+2, \quad \ncic{\hskip-1.2pt1\!0}:  r=2b, \quad
\ncic{\hskip-1.2pt1\!2}:r=2b+2
\end{equation*}
from the following picture which also shows all the other coaction cuts.
\begin{center}
\begin{tikzpicture}[scale=0.9]
\node (A0) at (0.05,0) {$\hskip-1.4cm 0;\eta_1,\ 0,\ldots,0,\ldots,\ 0,\eta_2i,\ 0,\ldots,0,\ldots,0,\ \eta_3,\ 0,\ 0;1$};
\node (A1) at (-4.0,-0.4) {${}$};
\node (A2) at (-4.0,0.35) {${}$};
\node (A3) at (-3.4,0.6) {${}$};
\node (A4) at (-3.5,-0.6) {${}$};
\node (A5) at (-0.6,-0.5) {${}$};
\node (A6) at (-1.7,0.5) {${}$};
\node (A7) at (-.7,0.7) {${}$};
\node (A8) at (-.2,-0.7) {${}$};
\node (A9) at (0.8,-0.8) {${}$};
\node (A10) at (1.0,0.82) {${}$};
\node (A11) at (1.4,-0.4) {${}$};
\node (A12) at (1.5,-0.98) {${}$};
\node (A13) at (2.2,0.6) {${}$};
\draw (-5.42,-0.25) to (-5.42,-0.4) to (A1) node {$\cic{1}$} to (-1.3,-0.4) to (-1.3,-0.25);
\draw (-5.43,0.25) to (-5.43,0.35) to (A2) node {$\cic{2}$} to (2.5,0.35) to (2.5,0.25);
\draw (-5.1,0.25) to (-5.1,0.6) to (A3) node {$\cic{3}$} to (0.57,0.6) to (0.58,0.25);
\draw (-5,-0.25) to (-5,-0.6) to (A4) node {$\cic{4}$} to (2.45,-0.6) to (2.45,-0.25);
\draw (-4.9,-0.25) to (-4.9,-0.5) to (A5) node {$\cic{5}$} to (3.1,-0.5) to (3.1,-0.25);
\draw (-5,0.25) to (-5,0.5) to (A6) node {$\cic{6}$} to (3.58,.5) to (3.58,0.25);
\draw (-4.9,0.25) to (-4.9,0.7) to (A7) node {$\cic{7}$} to (4.03,0.7) to (4.03,0.25);
\draw (-3.14,-0.25) to (-3.14,-0.7) to (A8) node {$\cic{8}$} to (2.5,-0.7) to (2.5,-0.25);
\draw (-3.18,-0.25) to (-3.18,-0.8) to (A9) node {$\cic{9}$} to (3.98,-0.8) to (3.98,-0.25);
\draw (-1.1,0.25) to (-1.1,0.82) to (A10) node {$\cic{\hskip-1.2pt1\!0}$}  to (3.06,0.82) to (3.06,0.25);
\draw (-1.1,-0.25) to (-1.1,-0.4) to (A11) node {$\cic{\hskip-1.2pt1\!1}$} to (3.62,-0.4) to (3.62,-0.25);
\draw (-1.2,-0.25) to (-1.2,-0.98) to (A12) node {$\cic{\hskip-1.2pt1\!2}$} to (4.03,-0.98) to (4.03,-0.25);
\draw (0.61,0.25) to (0.61,0.6) to (A13) node {$\cic{\hskip-1.2pt1\!3}$} to (3.98,0.6) to (3.98,0.25);
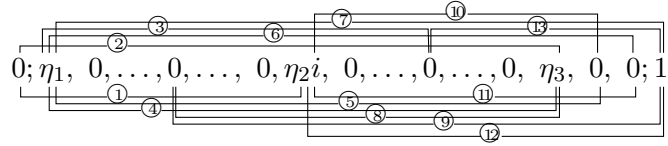
\node (A) at (0,-1.6) {Possible cuts of $D_r t^\fm(\ol{2a+1},\ol{2b},3)$.};
\end{tikzpicture}
\end{center}

Now we assume $r<w$ is odd. By direct calculations, we see that
\begin{align*}
D_r t^\fm (\ol{2a+1},\ol{2b},3)\equiv D_r \sum_{\eta_j\in\{\pm 1\}} \eta_1 I^\fm(0;\eta_1,0_{2a},\eta_2 \myi,0_{2b-1},\eta_3,0,0;1) =\ncic{1}+\cdots+\ncic{\hskip-1.2pt1\!3},
\end{align*}
where
\begin{align*}
\ncic{1}&=\delta_{r=2a+1} \sum_{\eta_j\in \{\pm 1\}} \eta_1 I^l(0;\eta_1,0_{2a};\eta_2 \myi)\ot I^m(0;\eta_2 \myi,0_{2b-1},\eta_3,0,0;1)\\
&=-\delta_{r=2a+1} \sum_{\eta\in\{\pm 1\}} \eta  I^l(0;\eta_1\myi,0_{2a};1)\ot \sum_{\eta_2,\eta_3\in\{\pm 1\}} \eta_2 I^m(0;\eta_2 \myi,0_{2b-1},\eta_3,0,0;1)\\
&=-\delta_{r=2a+1}   t^\fl(\ol{2a+1})  \ot\sum_{\eta_2,\eta_3\in\{\pm 1\}} \eta_2 I^m(0;\eta_2 \myi,0_{2b-1},\eta_3,0,0;1)=0 \quad(\text{by Cor.~\ref{cor:tr}}), \\
\ncic{2}&=\delta_{r=2a+2b+1} \sum_{\eta_j\in \{\pm 1\}} \eta_1 I^\fl(0;\eta_1,0_{2a},\eta_2 \myi,0_{2b-1};\eta_3)\ot I^\fm (0;\eta_3,0,0;1)\\
&=\delta_{r=2a+2b+1}  \sum_{\eta_1,\eta_2\in \{\pm 1\}}  \eta_1 I^\fl(0;\eta_1,0_{2a},\eta_2 \myi,0_{2b-1};1)\ot \eta_3 I^\fm (0;\eta_3,0,0;1) \\
&=\delta_{r=2a+2b+1}  t^\fl (\ol{2a+1},\ol{2b})\ot t^\fm (3)\in \calL_{2a+2b+1}\ot \calH_{3}  \quad(\text{by Theorem~\ref{thmone-doublealtert}}), \\
\ncic{3}&=\delta_{2a+1\leq r\leq 2a+2b-1} \sum_{\eta_j\in \{\pm 1\}} \eta_1 I^\fl (\eta_1;0_{2a},\eta_2 \myi,0_{r-2a-1};0)\ot I^\fm (0;\eta_1,0_{2a+2b-r},\eta_3,0,0;1)\\
&=-\delta_{2a+1\leq r\leq 2a+2b-1} \sum_{\eta_j\in \{\pm 1\}} \eta_1 I^\fl (0;0_{r-2a-1},\eta_2 \myi,0_{2a};1)\ot I^\fm (0;\eta_1,0_{2a+2b-r},\eta_3,0,0;1)\\
&=-\delta_{2a+1\leq r\leq 2a+2b-1}  \sum_{\eta\in \{\pm 1\}} \ze^\fl_{r-2a-1}  \begin{pmatrix} 2a+1\\ \eta \myi\end{pmatrix} \ot t^\fm (2a+2b+1-r,3)\\
&=-\delta_{2a+1\leq r\leq 2a+2b-1} \binom{r-1}{2a} \Big(\Li_r^\fl(\myi)+\Li_r^\fl(-\myi)\Big)\ot t^\fm (2a+2b+1-r,3) \\
&\in \calL_{r}\ot \calH_{w-r} \quad (\text{by \eqref{equ:calL_r}}),\\
\ncic{5}&=\delta_{ r=2a+2b+1} \sum_{\eta_j\in \{\pm 1\}} \eta_1 I^\fl (\eta_1;0_{2a},\eta_2 \myi,0_{2b-1},\eta_3;0)\ot I^\fm (0;\eta_1,0,0;1)\\
&=-\delta_{ r=2a+2b+1} \sum_{\eta_j\in \{\pm 1\}}  I^\fl (0;\eta_3,0_{2b-1},\eta_2 \myi,0_{2a};1)\ot \eta_1 I^\fm (0;\eta_1,0,0;1)\\
&=- \ze^\fl (\ol{2b},\ol{2a+1}) \ot t^\fm (3)
\in \calL_{r}\ot \calH_{w-r}  \quad(\text{by parity reduction}),\\
\ncic{7}&=\delta_{r=2a+2b+3} \sum_{\eta_j\in \{\pm 1\}} \eta_1 I^\fl(\eta_1;0_{2a},\eta_2 \myi,0_{2b-1},\eta_3,0,0;1)\ot I^\fm (0;\eta_1;1)\\
&=\delta_{r=2a+2b+3}\sum_{\eta_j\in \{\pm 1\}} \eta_1\Bigg(I^\fl(0;0_{2a},\eta_2 \myi,0_{2b-1},\eta_3,0,0;1)-I^\fl(0;0,0,\eta_3,0_{2b-1},\eta_2i,0_{2a};1)\Bigg)\ot I^\fm (0;\eta_1;1)\\
&=-\delta_{r=2a+2b+3}\sum_{\eta_1,\eta_2\in \{\pm 1\}} \bigg(\ze^\fl_2\binom{2b,2a+1}{\eta_1 \myi,\eta_2}-\ze_{2a}^\fl\binom{2b,3}{\eta_1 \myi,\eta_2} \bigg)\ot \log^\fm (2) =0 \quad(\text{by Lemma~\ref{lem:comb}}),\\
\ncic{8}&=\delta_{2b\leq r\leq 2a+2b-1}  \sum_{\eta_j\in \{\pm 1\}} \eta_1 I^\fl(0;0_{r-2b},\eta_2 \myi,0_{2b-1};\eta_3)\ot I^\fm(0;\eta_1,0_{2a+2b-r},\eta_3,0,0;1)\\
&=\delta_{2b\leq r\leq 2a+2b-1}  \sum_{\eta_j\in \{\pm 1\}} \eta_1 I^\fl(0;0_{r-2b},\eta_2 \myi,0_{2b-1};1)\ot I^\fm(0;\eta_1,0_{2a+2b-r},\eta_3,0,0;1)\\
&=\delta_{2b\leq r\leq 2a+2b-1}  \sum_{\eta\in \{\pm 1\}} \ze^\fl_{r-2b} \binom{2b}{\eta \myi} \ot t^\fm(2a+2b+1-r,3)\\
&=-\delta_{2b\leq r\leq 2a+2b-1} \binom{r-1}{2b-1}\Bigg(\ze^\fl\binom{r}{\eta \myi} +\ze^\fl\binom{r}{-\myi} \Bigg)\ot t^\fm(2a+2b+1-r,3) \\
& \in \calL_{r}\ot \calH_{w-r}  \quad(\text{by \eqref{equ:calL_r}}),\\
\ncic{9}&=\delta_{2b+3\leq r\leq 2a+2b+2}  \sum_{\eta_j\in \{\pm 1\}} \eta_1 I^\fl(0;0_{r-2b-3},\eta_2 \myi,0_{2b-1},\eta_3,0,0;1)\ot I^\fm (0;\eta_1,0_{2a+2b+3-r};1)\\
&=\delta_{2b+3\leq r\leq 2a+2b+2} \sum_{\eta_1,\eta_2\in \{\pm 1\}}\ze^\fl_{r-2b-3}  \binom{2b,3}{\eta_1 \myi,\eta_2}\ot t^\fm (2a+2b+4-r)\\
&=\delta_{2b+3\leq r\leq 2a+2b+2} \tz_{r-2b-3}^\fl(\ol{2b},3)\ot t^\fm (2a+2b+4-r)\\
&=\delta_{2b+3\leq r\leq 2a+2b+2} \sum_{i_1+i_2=r-2b-3} \binom{2b+i_1-1}{i_1}\binom{i_2+2}{i_2} \tz^\fl(\ol{2b+i_1},i_2+3)\ot t^\fm (2a+2b+4-r)\\
&\in \calL_{r}\ot \calH_{w-r} \quad(\text{by parity reduction}),\\
\ncic{\hskip-1.2pt1\!1}&=\delta_{r=2b+1}  \sum_{\eta_j\in \{\pm 1\}} \eta_1 I^\fl(\eta_2 \myi;0_{2b-1},\eta_3,0;0)\ot I^\fm (0;\eta_1,0_{2a},\eta_2 \myi,0;1)\\
&=-\delta_{r=2b+1}  \sum_{\eta_j\in \{\pm 1\}} \eta_1 I^\fl(0;0,\eta_3 \myi,0_{2b-1};1)\ot I^\fm (0;\eta_1,0_{2a},\eta_2 \myi,0;1)\\
&=-\delta_{r=2b+1}  \sum_{\eta \in \{\pm 1\}} \ze^\fl_1 \binom{2b}{\eta \myi}  \ot \sum_{\eta_1,\eta_2\in\{\pm 1\}} -\eta_1\eta_2 \ze^\fm  \binom{2a+1,2}{\eta_1 \myi,\eta_2 \myi}\\
&=\delta_{r=2b+1} 2b \sum_{\eta \in \{\pm 1\}} \ze^\fl \binom{2b+1}{\eta \myi} \ot t^\fm (\ol{2a+1},\bar2)\\
&=\delta_{r=2b+1} 2b \tz^\fl (\ol{2b+1})\ot t^\fm (\ol{2a+1},\bar2)\in \calL_{2b+1}\ot \calH_{2a+3}  \quad(\text{by Theorem~\ref{thmone-doublealtert}}), \\
\ncic{\hskip-1.2pt1\!3}&=\delta_{3\leq r\leq 2b+1}  \sum_{\eta_j\in \{\pm 1\}} \eta_1 I^\fl(0;0_{r-3},\eta_3,0,0;1)\ot I^\fm(0;\eta_1,0_{2a},\eta_2 \myi,0_{2b+2-r};1)\\
&=\delta_{3\leq r\leq 2b+1} \sum_{\eta\in\{\pm 1\}} \ze^\fl_{r-3}\binom{3}{\eta}\ot  \sum_{\eta_1\eta_2 \in \{\pm 1\}} -\eta_1\eta_2 \ze^\fm \binom{2a+1,2b+3-r}{\eta_1 \myi,\eta_2 \myi}\\
&=\delta_{3\leq r\leq 2b+1} (-1)^{r-1}\binom{r-1}{2} \tz^\fl(r)\ot t^\fm(\ol{2a+1},\ol{2b+3-r})
\in \calL_{r}\ot \calH_{w-r}
\end{align*}
by Theorem~\ref{thmone-doublealtert}. This completes the proof of the theorem by Theorem~\ref{thm:criterion}.
\end{proof}

\begin{lem}\label{lem:comb}
For all $a,b\in\N$ we have
\begin{equation*}
\sum_{\eta_1,\eta_2\in \{\pm 1\}} \bigg(\ze^\fl_2\binom{2b,2a+1}{\eta_1 \myi,\eta_2}-\ze_{2a}^\fl\binom{2b,3}{\eta_1 \myi,\eta_2} \bigg)=0.
\end{equation*}
\end{lem}

\begin{proof}
By definition,
\begin{align*}
\ze^\fl_2\binom{2b,2a+1}{\eta_1 \myi,\eta_2}
=&\,\sum_{i_1+i_2=2,\atop i_1,i_2\ge 0} \binom{2b+i_1-1}{i_1}\binom{2a+i_2}{i_2} \ze^\fl\binom{2b+i_1,2a+1+i_2}{\eta_1 \myi,\eta_2},\\
\ze_{2a}^\fl\binom{2b,3}{\eta_1 \myi,\eta_2}=&\,\sum_{i_1+i_2=2a,\atop i_1,i_2\ge 0} \binom{2b+i_1-1}{i_1}\binom{2+i_2}{i_2} \ze^\fl\binom{2b+i_1,3+i_2}{\eta_1 \myi,\eta_2}.
\end{align*}
By Panzer's parity formula for double polylogarithms \cite[(3.2)]{Panzer2017}, setting $w=2a+2b+3$, $z_1=\myi$ and $z_2=\eta_2$, we get
\begin{equation*}
\sum_{\eta_1\in \{\pm 1\}} \ze^\fl\binom{2b+i_1,2a+1+i_2}{\eta_1 \myi,\eta_2}
= (-1)^{i_1} \binom{w-1}{2b+i_1-1}\Li^\fl_w(\myi)-\Li^\fl_w(\eta_2 \myi)
-(-1)^{i_2} \binom{w-1}{2a+1+i_2}\Li^\fl_w(\eta_2).
\end{equation*}
Similarly,
\begin{equation*}
\sum_{\eta_1\in \{\pm 1\}} \ze^\fl\binom{2b+i_1,3+i_2}{\eta_1 \myi,\eta_2}
= (-1)^{i_1} \binom{w-1}{2b+i_1-1}\Li^\fl_w(\myi)-\Li^\fl_w(\eta_2 \myi)
-(-1)^{i_2} \binom{w-1}{3+i_2}\Li^\fl_w(\eta_2).
\end{equation*}
Since
\begin{equation*}
\sum_{\eta_2\in \{\pm 1\}} \Li^\fl_w(\eta_2 \myi) =0
\end{equation*}
by \eqref{equ:calL_r} we see that we only need to prove the following two identities:
\begin{multline}\label{equ:comb1}
\sum_{i_1+i_2=2,\atop i_1,i_2\ge 0} \binom{2b+i_1-1}{i_1}\binom{2a+i_2}{i_2}(-1)^{i_2} \binom{w-1}{2a+1+i_2} \\
=\sum_{i_1+i_2=2a,\atop i_1,i_2\ge 0} \binom{2b+i_1-1}{i_1}\binom{2+i_2}{i_2}(-1)^{i_2} \binom{w-1}{3+i_2},
\end{multline}
and
\begin{multline}\label{equ:comb2}
\sum_{i_1+i_2=2,\atop i_1,i_2\ge 0} \binom{2b+i_1-1}{i_1}\binom{2a+i_2}{i_2}(-1)^{i_1} \binom{w-1}{2b+i_1-1}\\
=\sum_{i_1+i_2=2a,\atop i_1,i_2\ge 0} \binom{2b+i_1-1}{i_1}\binom{2+i_2}{i_2}(-1)^{i_1} \binom{w-1}{2b+i_1-1}.
\end{multline}

To prove them, we set for convenience
\begin{equation*}
A=2a,\qquad B=2b
\end{equation*}
are positive even integers and $w=A+B+3$. It is straightforward to check that left-hand side of \eqref{equ:comb1} equals
\begin{equation}\label{equ:LHSofComb1}
L:=\frac{B(B+1)}{(A+2)(A+3)}\binom{A+B+2}{A+1}.
\end{equation}
Put $i_2=k$ and $i_1=A-k$. The right-hand side of \eqref{equ:comb1} is
\begin{equation*}
R:=\sum_{k=0}^A \binom{A+B-k-1}{A-k}\binom{k+2}{k}(-1)^k \binom{A+B+2}{k+3}.
\end{equation*}
A direct computation gives
\begin{equation*}
R=\binom{A+B-1}{A}\binom{A+B+2}{3}{}_2F_1(-A,3;4;1).
\end{equation*}
By Chu--Vandermonde for the Gauss hypergeometric function \cite[\S15.4.24]{Abramowitz}, we get
\begin{equation*}
{}_2F_1(-A,3;4;1)=\frac{(4-3)_A}{(4)_A}=\frac{6}{(A+1)(A+2)(A+3)}.
\end{equation*}
Hence,
\begin{equation*}
R=\binom{A+B-1}{A}\binom{A+B+2}{3}\frac{6}{(A+1)(A+2)(A+3)}=L
\end{equation*}
which is the left-hand side of \eqref{equ:comb1} as given by \eqref{equ:LHSofComb1}.

For \eqref{equ:comb2}, a direct computation shows that the left-hand side of \eqref{equ:comb2} is equal to
\begin{equation*}
     L:=\binom{A+B+2}{B-1}.
\end{equation*}
We now put $i_1=k$ and $i_2=A-k$. The right-hand side of \eqref{equ:comb1} is
\begin{equation*}
R=\sum_{k=0}^{A}\binom{B+k-1}{k}\binom{A+2-k}{2}(-1)^k\binom{A+B+2}{B+k-1}.
\end{equation*}
A straightforward simplification gives
\begin{equation*}
  \frac{R}{L}=\sum_{k=0}^{A} (-1)^k\binom{A+3}{k}\binom{A+2-k}{2}
\end{equation*}
which we will show is equal to 1.  First, we have
\begin{equation*}
  \frac{R}{L}-1=\sum_{k=0}^{A+3} (-1)^k\binom{A+3}{k}\binom{A+2-k}{2}
\end{equation*}
since terms for $k=A+1, A+2$ are zero and the term for $k=A+3$ is $-1$. Now,
\begin{equation*}
\binom{A+2-k}{2}=\Coeff_{x^2} \Big\{(1+x)^{A+2-k}\Big\}.
\end{equation*}
Hence
\begin{align*}
\frac{R}{L}-1
&=\Coeff_{x^2} \bigg\{\sum_{k=0}^{A+3}(-1)^k\binom{A+3}{k}(1+x)^{A+2-k} \bigg\}\\
&=\Coeff_{x^2} \bigg\{(1+x)^{A+2}\sum_{k=0}^{A+3}\binom{A+3}{k}\left(-\frac1{1+x}\right)^k\bigg\}\\
&=\Coeff_{x^2} \bigg\{(1+x)^{A+2}\left(1-\frac1{1+x}\right)^{A+3}\bigg\}\\
&=\Coeff_{x^2} \bigg\{(1+x)^{A+2}\left(\frac{x}{1+x}\right)^{A+3}\bigg\}\\
&=\Coeff_{x^2} \bigg\{\frac{x^{A+3}}{1+x}\bigg\}=0
\end{align*}
since $A=2a\ge2$. This implies that $R=L$ which concludes the proof of the lemma.
\end{proof}

In fact, from Theorems 36 and 40 in \cite{Xu-Wang2023}, it follows that $t(\ol{2a+1},\ol{2b},3)$ can be expressed as a $\Q$-linear combination of alternating double $t$-values, alternating double zeta values, and product of single zeta values. For example:
\begin{align*}
t(\bar3,\bar2,3)=-\frac{31}{64}\ze(3)\ze(5)+\frac{135}{512}\ze(2)\ze(3)^2+\frac{221}{320}\ze(3,5)-\frac{387}{5120}\ze(8)-t(\bar3,\bar5).
\end{align*}
Furthermore, by applying Theorem \ref{thmone-doublealtert} regarding the unramifiedness of alternating double $t$-values, together with the unramifiedness result for motivic Euler sums established in \cite[Theorem 6.2.1]{DeligneGo2005}, we can further deduce that $t(\bar3,\bar2,3)$ is unramified. Analogous discussions work for the general cases $t(\ol{2a+1},\ol{2b},3)$, too.

\section{Unramified motivic alternating multiple $T$-values}\label{sec:moreAltV}

In general, among the alternating multiple mixed values
\begin{align*}
M(k_1,\ldots,k_d;\varepsilon_1,\ldots,\varepsilon_d;\sigma_1,\ldots,\sigma_d),
\end{align*}
those with $\varepsilon_j=(-1)^j$ are called alternating \emph{multiple $T$-values} (MTVs). The non-alternating MTVs are first introduced and systematically studied by Kaneko and Tsumura \cite{KanekoTs2020}, whose motivic version is given by (see \cite{Murakami2021}):
\begin{align}\label{defn-MTV-motivic}
T^\fm(k_1,\ldots,k_d)=\sum_{\eta_j\in\{\pm 1\}} \eta_1\cdots\eta_d I^\fm(0;\eta_1,0_{k_1-1},\ldots,\eta_d,0_{k_d-1};1).
\end{align}
In our previous work \cite{XuZhao2024,XuYanZhao2022Aug}, we investigated several special values and algebraic relations for alternating MTVs. In particular, we found that when all $k_j=1$, $\sigma_d=-1$, and $\sigma_j=1$ for the remaining indices, one has
\begin{align}\label{equ:olTvalue}
\olT(\{1\}_d):=T(\{1\}_d;\{1\}_{d-1},-1)=\frac{(-1)^{[(d+1)/2]}}{d!}\Big(\frac{\pi}{2}\Big)^d\in \pi^d \Q ,
\end{align}
where $\{1\}_d$ stands for the string of 1's repeated $d$ times. Observe that \eqref{defn-AMMV-integrals} implies
\begin{align}\label{defn-AMTV}
\olT(\{1\}_d)=(-\myi)^{\delta_{2\nmid d}}\sum_{\eta_j\in\{\pm 1\}} \eta_1\cdots\eta_d I(0;\eta_1\myi,\ldots,\eta_d\myi;1).
\end{align}

Moreover, our numerical computation indicates that, apart from the known fact that
\begin{align}\label{equ:ANAolTfamily}
\olT(\{1\}_{2d})=\frac{(-1)^d}{(2d)!}\Big(\frac{\pi}{2}\Big)^{2d}
\end{align}
is unramified, all other alternating MTVs appear to be ramified.
Further, the equality \eqref{defn-AMTV} can also be promoted to the motivic level with the aid of the next lemma.

\begin{lem} \label{lem:T1d}
For all $d\in \N$, we have
\begin{align*}
T^\fm(\{1\}_d)=\frac{(\log^\fm 2)^d}{d!}.
\end{align*}
\end{lem}

\begin{proof}
We prove the lemma by induction on $d$. If $d=1$ then by \eqref{defn-MTV-motivic}
\begin{align*}
T^\fm(1)=I^\fm(0;1;1)-I^\fm(0;-1;1)=\log^\fm 2.
\end{align*}
Suppose the lemma hold for all depths $<d$. Then for any odd $r<d$ the coaction
\begin{align*}
&\, D_r T^\fm(\{1\}_d)=\sum_{\eta_j\in\{\pm 1\}} \eta_1\cdots \eta_d D_r I^\fm(0;\eta_1,\ldots,\eta_d;1)\\
=&\, \sum_{\eta_j\in\{\pm 1\}} \eta_1\cdots \eta_d
I^\fl(0;\eta_1,\ldots,\eta_r;\eta_{r+1})\ot I^\fm(0;\eta_{r+1},\ldots,\eta_d;1)\\
+&\, \sum_{k=1}^{d-r-1} \sum_{\eta_j\in\{\pm 1\}} \eta_1\cdots \eta_d
I^\fl(\eta_k;\eta_{k+1},\ldots,\eta_{k+r};\eta_{k+r+1})\\
&\, \hskip5cm \ot I^\fm(0;\eta_1,\ldots,\eta_k,\eta_{k+r+1},\ldots,\eta_d;1) \\
+&\, \sum_{\eta_j\in\{\pm 1\}} \eta_1\cdots \eta_d
I^\fl(\eta_{d-r};\eta_{d-r+1},\ldots,\eta_d;1)\ot I^\fm(0;\eta_1,\ldots,\eta_{d-r};1)\\
=&\, \sum_{\eta_j\in\{\pm 1\}} \eta_1\cdots \eta_r
I^\fl(0;\eta_1,\ldots,\eta_r;1)\ot  \eta_{r+2}\cdots \eta_d I^\fm(0;\eta_{r+1},\ldots,\eta_d;1)\\
+&\, \sum_{k=1}^{d-r-1} \sum_{\eta_j\in\{\pm 1\}} \eta_1\cdots \eta_d
\Big[I^\fl(0;\eta_{k+1} ,\ldots,\eta_{k+r} ;\eta_{k+r+1})-I^\fl(0;\eta_{k+1} ,\ldots,\eta_{k+r};\eta_k)\Big]\\
&\, \hskip5cm \ot I^\fm(0;\eta_1,\ldots,\eta_k,\eta_{k+r+1},\ldots,\eta_d;1) \\
+&\, \sum_{\eta_j\in\{\pm 1\}} \eta_1\cdots \eta_d
\Big[I^\fl(\eta_{d-r} ;\eta_{d-r+1} ,\ldots,\eta_d ;0)+I^\fl(0;\eta_{d-r+1},\ldots,\eta_d;1)\Big]\ot I^\fm(0;\eta_1,\ldots,\eta_{d-r};1) \\
=&\, T^\fl(\{1\}_r)\ot  U^\fm_{d-r}(1)
+ \sum_{k=1}^{d-r-1} T^\fl(\{1\}_r)\ot \Big[ U^\fm_{d-r}(k+1) -U^\fm_{d-r}(k)  \Big]  - T^\fl(\{1\}_r)\ot  U^\fm_{d-r}(d-r)\\
&\, +  T^\fl(\{1\}_r)\ot T^\fm(\{1\}_{d-r})
\end{align*}
where
\begin{equation*}
    U^\fm_n(j) =\sum_{\eta_j\in\{\pm 1\}} \eta_1\cdots\eta_{j-1}\eta_{j+1}\cdots \eta_n I^\fm(0;\eta_1,\ldots,\eta_n;1).
\end{equation*}
By induction assumption we see that $T^\fl(\{1\}_r)=0$ for all $r>1$. On the other hand, by telescoping and induction again we get
\begin{align*}
&\, D_1 T^\fm(\{1\}_d)=T^\fl(1)\ot T^\fm(\{1\}_{d-1})=\log^\fl 2\ot \frac{(\log^\fm 2)^{d-1}}{(d-1)!}=D_1 \frac{(\log^\fm 2)^d}{d!}.
\end{align*}
Since $D_r (\log^\fm 2)^d=0$ for all $r>1$, by Brown's criterion theorem (or just use Theorem~\ref{thm:criterion1}) we must have
\begin{equation*}
     T^\fm(\{1\}_d)=\frac{(\log^\fm 2)^d}{d!}+c \zeta^\fm(d)
\end{equation*}
for some $c\in\Q$. However, by taking the regularization of $T^\fm(\{1\}_d)$ we see that
\begin{equation*}
\int_0^{1-\eps} \left(\frac{2 dt}{1-t^2}\right)^d=\frac{1}{d!} \left( \int_0^{1-\eps} \frac{2 dt}{1-t^2}\right)^d
=\frac{1}{d!}  \left(\log \frac{2-\eps}{\eps}\right)^d
=\frac{1}{d!}  \left(\log 2-\log \eps\right)^d +O(\eps \log^{d-1} \eps).
\end{equation*}
Thus
\begin{equation*}
    \dch( T^\fm(\{1\}_d))=\frac{(\log 2)^d}{d!}=\frac{(\log 2)^d}{d!}+c \zeta^\fm(d)
\end{equation*}
which implies that $c=0$. This completes the proof of the lemma.
\end{proof}

\begin{thm}
For all $d\in \N$, we have
\begin{align*}
\olT^\fm(\{1\}_d)=\frac{\myi^{\delta_{2\nmid d}}}{d!}\Big(\frac{ (2\pi \myi)^\fm } {4}\Big)^d
\end{align*}
where $(2\pi \myi)^\fm=4\Li^\fm_1(\myi)-4\Li^\fm_1(-\myi)=4I^\fm(0;-\myi;1)-4I^\fm(0;\myi;1)$ is the Lefschetz motive.
\end{thm}
\begin{proof}
First, from the definition of the period map we see that
\begin{align*}
\dch\Big((2\pi \myi)^\fm\Big)&=4\Li_1(\myi)-4\Li_1(-\myi)\\
&=4\big(I(0;-\myi;1)-I(0;\myi;1)\big)
=4\int_0^1\left(\frac{1}{-\myi-t}-\frac{1}{\myi-t}\right)\, dt=2\pi \myi.
\end{align*}
Thus, $(2\pi \myi)^\fm$ is indeed the motivic avatar for $2\pi\myi$.

We now proceed by induction on $d$. From the definition of the alternating MTVs \eqref{defn-AMTV} together with the definition of motivic alternating multiple mixed values \eqref{defn-AMMV-Motivic}, we obtain
\begin{align*}
\olT^\fm(\{1\}_d)=(-\myi)^{\delta_{2\nmid d}}\sum_{\eta_j\in\{\pm 1\}} \eta_1\cdots \eta_d I^\fm(0;\eta_1 \myi,\ldots,\eta_d \myi;1).
\end{align*}
When $d=1$, we see immediately that
\begin{align*}
\olT^\fm(1)=-\myi\big(I^\fm(0;\myi;1)-I^\fm(0;-\myi;1)\big)=\myi \frac{ (2\pi \myi)^\fm } {4}.
\end{align*}

Suppose $d\ge 2$. Let $r<w$ be a positive integer. Then
\begin{align*}
&\,  \myi^{\delta_{2\nmid d}}D_r\olT^\fm(\{1\}_d)=\sum_{\eta_j\in\{\pm 1\}} \eta_1\cdots \eta_d D_r I^\fm(0;\eta_1 \myi,\ldots,\eta_d \myi;1)\\
=&\, \sum_{\eta_j\in\{\pm 1\}} \eta_1\cdots \eta_d
I^\fl(0;\eta_1 \myi,\ldots,\eta_r \myi;\eta_{r+1} \myi)\ot I^\fm(0;\eta_{r+1} \myi,\ldots,\eta_d \myi;1)\\
+&\, \sum_{k=1}^{d-r-1} \sum_{\eta_j\in\{\pm 1\}} \eta_1\cdots \eta_d
I^\fl(\eta_k \myi;\eta_{k+1} \myi,\ldots,\eta_{k+r} \myi;\eta_{k+r+1} \myi)\\
&\, \hskip5cm \ot I^\fm(0;\eta_1 \myi,\ldots,\eta_k \myi,\eta_{k+r+1} \myi,\ldots,\eta_d \myi;1) \\
+&\, \sum_{\eta_j\in\{\pm 1\}} \eta_1\cdots \eta_d
I^\fl(\eta_{d-r} \myi;\eta_{d-r+1} \myi,\ldots,\eta_d \myi;1)\ot I^\fm(0;\eta_1 \myi,\ldots,\eta_{d-r} \myi;1)\\
=&\, \sum_{\eta_j\in\{\pm 1\}} \eta_1\cdots \eta_r
I^\fl(0;\eta_1,\ldots,\eta_r;1)\ot \eta_{r+1}^r \eta_{r+1}\cdots \eta_d I^\fm(0;\eta_{r+1} \myi,\ldots,\eta_d \myi;1)\\
+&\, \sum_{k=1}^{d-r-1} \sum_{\eta_j\in\{\pm 1\}} \eta_1\cdots \eta_d
I^\fl(0;\eta_{k+1} ,\ldots,\eta_{k+r} ;\eta_{k+r+1})\\
&\, \hskip5cm \ot I^\fm(0;\eta_1 \myi,\ldots,\eta_k \myi,\eta_{k+r+1} \myi,\ldots,\eta_d \myi;1) \\
+&\, \sum_{k=1}^{d-r-1} \sum_{\eta_j\in\{\pm 1\}} \eta_1\cdots \eta_d
I^\fl(\eta_k;\eta_{k+1} ,\ldots,\eta_{k+r} ;0)\\
&\, \hskip5cm \ot I^\fm(0;\eta_1 \myi,\ldots,\eta_k \myi,\eta_{k+r+1} \myi,\ldots,\eta_d \myi;1) \\
+&\, \sum_{\eta_j\in\{\pm 1\}} \eta_1\cdots \eta_d
I^\fl(\eta_{d-r} ;\eta_{d-r+1} ,\ldots,\eta_d ;0)\ot I^\fm(0;\eta_1 \myi,\ldots,\eta_{d-r} \myi;1)\\
+&\, \sum_{\eta_j\in\{\pm 1\}} \eta_1\cdots \eta_d
I^\fl(0;\eta_{d-r+1} \myi,\ldots,\eta_d \myi;1)\ot I^\fm(0;\eta_1 \myi,\ldots,\eta_{d-r} \myi;1).
\end{align*}
Notice that all the left factors in the above except for the last sum is essentially (up to a sign)
\begin{equation*}
     \sum_{\eta_j\in\{\pm 1\}} \eta_1\cdots \eta_r I^\fl(0;\eta_1,\ldots,\eta_r;1) =T^\fl(\{1\}_r)=0
\end{equation*}
for all $r\ge 1$ by Lemma~\ref{lem:T1d} and \eqref{equ:2pifl=0}. Hence, for all $r\ge 2$ we obtain by induction
\begin{align*}
\myi^{\delta_{2\nmid d}}D_r\olT^\fm(\{1\}_d)
=&\,   \myi^{\delta_{2\nmid r}} \olT^\fl(\{1\}_r) \ot \myi^{\delta_{2\nmid d-r}} \olT^\fm(\{1\}_{d-r})=0.
\end{align*}
Thus, we only need to consider $r=1$. Setting
\begin{equation*}
V^\fm_n(j)=\sum_{\eta_j\in\{\pm 1\}} \eta_1\cdots\eta_{j-1}\eta_{j+1}\cdots \eta_n I^\fm(0;\eta_1 \myi,\ldots,\eta_n \myi;1),
\end{equation*}
we have
\begin{align*}
\myi^{\delta_{2\nmid d}}D_1\olT^\fm(\{1\}_d)=&\,  \log^\fl 2 \ot  V^\fm_{d-1}(1)
+ \sum_{k=1}^{d-2}\log^\fl 2 \ot \big(V^\fm_{d-1}(k+1) -  V^\fm_{d-1}(k) \big) \\
&\, -\log^\fl 2 \ot   V^\fm_{d-1}(d-1)
+ \myi \olT^\fm(1) \ot  \myi^{\delta_{2\nmid d-1}}\olT^\fm(\{1\}_{d-1})\\
=&\,  - \frac{ (2\pi \myi)^\fl } {4} \ot \frac{(-1)^{\delta_{2\nmid d-1}} }{(d-1)!}\Big(\frac{ (2\pi \myi)^\fm } {4}\Big)^{d-1}=0.
\end{align*}
Hence, by the power rule
\begin{align*}
D_1\olT^\fm(\{1\}_d)= D_1 \frac{\myi^{\delta_{2\nmid d}}}{d!}\Big(\frac{ (2\pi \myi)^\fm } {4}\Big)^d=0.
\end{align*}
Therefore, the theorem follows readily from Theorem~\ref{thm:criterion1} since its analytic version \eqref{equ:olTvalue} holds.
\end{proof}

\begin{cor} \label{cor:olTfamily}
For all $d\in \N$, we have
\begin{align}\label{equ:olTfamily}
\olT^\fm(\{1\}_{2d})=-\frac{\ze^\fm(2d)}{2^{4d-1}B_{2d}},
\end{align}
where $B_{2d}$ are the Bernoulli numbers.
\end{cor}
\begin{proof}
Since $D_r=0$ for all $r\ge 1$ on both sides of \eqref{equ:olTfamily} the corollary follows easily from the analytic version \eqref{equ:ANAolTfamily} of \eqref{equ:olTfamily} by Euler's identity $2\zeta(2d)=-(-1)^d (2\pi)^{2d} B_{2d}/(2d)!.$
\end{proof}

\section{Three more special unramified families of alternating MMVs}
In this short section, we consider the following three special families of alternating motivic MMVs:
\begin{align*}
&M^\fm(\{\check{1},1\}_n,\check{\bar{1}},\check{\bar{1}})
:=M^\fm(\{1\}_{2n+2};\{-1,1\}_n,-1,-1;\{1\}_{2n},-1,-1)\\
&=\sum_{\eta_j\in\{\pm 1\}} \eta_1 \ldots \eta_{2n+1} I^\fm(0;\eta_1,\ldots,\eta_{2n+1},\eta_{2n+2}\myi;1),\\
&M^\fm(\{\check{1},1\}_n,\check{1},\bar{1},\bar{1})
:=M^\fm(\{1\}_{2n+3};\{-1,1\}_{n+1},1;\{1\}_{2n+1},-1,-1)\\
&=\sum_{\eta_j\in\{\pm 1\}} \eta_1 \ldots \eta_{2n+2} I^\fm(0;\eta_1,\ldots,\eta_{2n+2},\eta_{2n+3}\myi;1),\\
&M^\fm(\{\check{1},1\}_m,\check{\bar{1}},\check{\bar{1}},\{1,\check{1}\}_n,1,\check{2})\\
&:=M^\fm(\{1\}_{2m+2n+3},2;\{-1,1\}_m,\{-1\}_2,\{1,-1\}_{n+1};\{1\}_{2m},\{-1\}_2,\{1\}_{2n+2})\\
&=\sum_{\eta_j\in\{\pm 1\}} \eta_1\cdots \eta_{2m+1} \eta_{2m+3}\cdots \eta_{2m+2n+4} I^\fm(0;\eta_1,\ldots,\eta_{2m+1},\eta_{2m+2}\myi,\eta_{2m+3},\ldots,\eta_{2m+2n+4},0;1).
\end{align*}
Applying the duality relations of alternating MMVs, namely $t\rightarrow \frac{1-t}{1+t}$, yields
\begin{align*}
&\sum_{\eta_j\in\{\pm 1\}} \eta_1 \ldots \eta_{m-1} I(0;\eta_1,\ldots,\eta_{m-1},\eta_{m}\myi;1)
=\int_0^1 \left(\frac{2dt}{1-t^2}\right)^{m-1} \frac{-2tdt}{1+t^2}  \\
&=-\int_0^1 \left(\frac{2dt}{1-t^2}-\frac{2tdt}{1-t^2}-\frac{2tdt}{1+t^2} \right)\left(\frac{dt}{t}\right)^{m-1}\\
&=-\frac{2\cdot 4^{m-1}-3\cdot 2^{m-1}+1}{4^{m-1}}\zeta(m)
\end{align*}
and
\begin{align*}
&M(\{\check{1},1\}_m,\check{\bar{1}},\check{\bar{1}},\{1,\check{1}\}_n,1,\check{2})=\int_0^1 \left(\frac{2dt}{1-t^2}\right)^{2m+1} \frac{-2tdt}{1+t^2} \left(\frac{2dt}{1-t^2}\right)^{2n+2}\frac{dt}{t} \\
&=-\int_0^1 \frac{2dt}{1-t^2} \left(\frac{dt}{t}\right)^{2n+2} \left(\frac{2dt}{1-t^2}-\frac{2tdt}{1-t^2}-\frac{2tdt}{1+t^2} \right)\left(\frac{dt}{t}\right)^{2m+1}\\
&=-T(2n+3,2m+2)+t(2n+3,2m+2)-t(\overline{2n+3},\overline{2m+2}).
\end{align*}

Hence, the following theorem holds by the parity principle since all the above double values have odd weight $2m+2n+5$. We leave the
details of the proof to the interested reader.
\begin{thm}\label{thm:3specialFamilies}
For $n,m\in\mathbb{N}_0$, the following three families of alternating motivic MMVs, namely
\begin{align*}
M^\fm(\{\check{1},1\}_n,\check{\bar{1}},\check{\bar{1}}),\ M^\fm(\{\check{1},1\}_n,\check{1},\bar{1},\bar{1}), \text{and}\ M^\fm(\{\check{1},1\}_m,\check{\bar{1}},\check{\bar{1}},\{1,\check{1}\}_n,1,\check{2})
\end{align*}
are all unramified.
\end{thm}

We conclude our paper with the following conjecture based on our extensive numerical experiments.

\begin{conj}
All the unramified truly alternating MMVs are given by the alternating double $t$-values in Theorem~\ref{thmone-doublealtert}, the alternating triple $t$-values in Theorem~\ref{thm:tripleAltMtVs}, the alternating MTVs in Corollary~\ref{cor:olTfamily}, and the three special families in Theorem~\ref{thm:3specialFamilies}.
\end{conj}

{\bf Declaration of competing interest.}
The authors declare that they have no known competing financial interests related to this paper.

{\bf Data availability.}
No data were used for the research described in the article.

{\bf Acknowledgments.} This work was started and completed while both authors are visiting the Tianyuan Mathematical Center in Southeast China, supported by the funding number 12526102. They thank their hosts Prof. L. Lai and Prof. H. Zhu for the invitation. C. Xu is supported by the General Program of Natural Science Foundation of Anhui Province (Grant No. 2508085MA014). J. Zhao was supported by the Jacobs Prize from the Bishop's School.

\end{document}